\documentclass[12pt]{article}

\usepackage[top=0.75in]{geometry}
\usepackage{amsmath,amsthm,amssymb}
\usepackage{enumerate}
\usepackage{graphicx}

\newtheorem{thm}{Theorem}

\newtheorem{cor}[thm]{Corollary}

\theoremstyle{definition}
\newtheorem*{defi}{Definition}
\newtheorem*{rem}{Remark}

\newcommand{\abs}[1]{\left\vert#1\right\vert}

\begin{document}

\title{Reviews of Symbolic Moment Calculus}
\author{Thotsaporn Aek Thanatipanonda\\
Mahidol University International College \\
Nakhon Pathom, Thailand}
\date{March 26, 2020}

\maketitle

\begin{abstract}
As a former engineering student, I have a great
interest in a real world application of mathematics. Probability
is something I can relate to. I am lucky enough that
after I switched to Mathematics, this is one of 
many interests of my Ph.D. advisor, Doron Zeilberger, as well. 
In this article we create a program to apply the 
\textit{overlapping stage approach}
to calculate the moments $E[X^r]$ and 
$E[(X-\mu)^r]$ of combinatorial objects. 
We also show the normality property of their
distributions when these moments are easy 
enough to calculate.
\end{abstract}

The programs accompanied this article are 
\texttt{SchurT.txt}, \texttt{InvMaj.txt}, \texttt{Boolean.txt}
which can be found from the author's website:
\texttt{thotsaporn.com}.

\section{Introduction}
The expectation functional, $E[X]$ is a powerful tool to study combinatorial objects,
and often gives us quite useful information.
For example, the common technique in probabilistic method
to show the existence of the objects is to show $E[X] < 1$, i.e.
  \[E[X]<1 \text{ implies } Pr(X = 0) >0.\]
For the higher moments, $E[X^r]$,
the computation gets complicated fast and we need computers to do symbolic
computation. The main technique to deal 
with the higher moment is called
\textit{the overlapping stage approach}
and has already been demonstrated in \cite{Z2},
\cite{Z3}. We briefly describe this approach below.

\textbf{Calculation of Moments}

It is easy to see that $E[X^0]=1$.
The computation for $E[X]$ is somewhat harder.
The other higher moments than that are very 
hard to do without a computer. 

Write
\[ X = \sum_{S} X_S , \]

where the sum is over all the objects of interest $S$ over some domain,
and $X_S$ is the {\it indicator} random
variable that is 1 if the object $S$ satisfies the assigned property and
0 otherwise. 

We also have
\[
X^r = \sum_{i_1=1}^n\sum_{i_2=1}^n ... \sum_{i_r=1}^n X_{i_1}X_{i_2}...X_{i_r},
\]

and by {\it linearity of expectation} and {\it symmetry}, we further have

\[
E[X^r] = \sum_{i_1=1}^n\sum_{i_2=1}^n ... \sum_{i_r=1}^n E[X_{i_1}X_{i_2}...X_{i_r}],
\]

The difficulty of calculating the higher moment lies 
in how these random objects intersect with each other.

There are some former works of Doron Zeilberger as in
\cite{Z1,Z2,Z3,Z4,Z5,Z6} that we can study.  
For brute force calculation referred to \cite{Z1}. 
For overlapping stage approach referred to \cite{Z2,Z3}. 
For asymptotically normal on different statistics of permutations
refer to \cite{Z5,Z6}. This work can be viewed as a continuation of \cite{Z2,Z3}. 

Apart from the mathematical contents we are presenting, 
the goals (in a bigger picture) of this article are to
\begin{enumerate}
\item Serve as an introductory reading for an
interested reader.
\item Give examples of combinatorial objects that
the moments can be computed from. We consider
Schur triples, inversion and major index, Boolean boards
and Boolean functions.
\item Demonstrate an application of symbolic computation
which, in this case, was done by Maple program.
\end{enumerate}

\section{Monochromatic Schur Triples on $[1,n]$}

Monochromatic Schur Triples on $[1,n]$ is a topic in Ramsey theory.
We let $\mathcal{U}$ be set of all $c$-integer-colorings
of $[1,n]$, and $X := X(u)$ be the number of monochromatic Schur triples
$S := \{x,y,x+y\}$ of $[1,n]$ in $u \in \mathcal{U}$.
We calculate $E[X^r]$, the $r^{th}$ moment of $X.$

\textbf{The first moment:} 

A triple $S$ can be written as
$S = S_1 \cup S_2$ where $S_1$ are triples of the form $\{x,x,2x\}$
and $S_2$ are triples of the form $\{x,y,x+y\}, x \neq y$.
\begin{align*}
E[X_{S_1}] &= c \cdot \dfrac{1}{c^2} = \dfrac{1}{c}, \\
E[X_{S_2}] &= c \cdot \dfrac{1}{c^3} = \dfrac{1}{c^2}.
\end{align*}

Already there are two formulas for $E[X]$
depending on whether
the length of the interval, $n$, is odd or even. 

\noindent \textbf{Case 1: $n$ is odd}
\begin{align*}
E[X] &= \sum_{S} E[X_S] \\
&= \sum_{S_1} E[X_{S_1}]+\sum_{S_2} E[X_{S_2}] \\
&= \frac{1}{c}\left(\frac{n-1}{2}\right) + \frac{1}{c^2}( (n-2) +(n-4)+(n-6)+...+1) \\
&= \frac{1}{c}\left(\frac{n-1}{2}\right) + \frac{1}{c^2}\left( \frac{n-1}{2}\right)^2 \\
&= \frac{(n-1)(n-1+2c)}{4c^2}.
\end{align*}
\textbf{Case 2: $n$ is even}
\begin{align*}
E[X] &= \sum_{S_1} E[X_{S_1}]+\sum_{S_2} E[X_{S_2}] \\
&= \frac{1}{c}\left(\frac{n}{2}\right) + 
\frac{1}{c^2}( (n-2) +(n-4)+(n-6)+...+0) \\
&= \frac{1}{c}\left(\frac{n}{2}\right) + 
\frac{1}{c^2}\left(\frac{n}{2}\right)\left(\frac{n}{2}-1\right) \\
&= \frac{n(n-2+2c)}{4c^2}.
\end{align*}

\textbf{The second moment:} 

The calculation for the second moment is considerably harder.
We consider $2$-tuples $[S_1,S_2]$ of triples $\{x,y,x+y\}$ of $[1,n].$ 

\[ E[X^2] = E\left[\sum_{S_1}X_{S_1} \cdot \sum_{S_2}X_{S_2}\right] 
= \sum_{[S_1,S_2]}E[X_{S_1}X_{S_2}].\]

The sum depends on how $S_1$ and $S_2$ interact with each other.
For each configuration of $S_1 \cup S_2, \;\ 2 \leq \abs{S_1 \cup S_2} \leq 6.$
For example,  $\abs{S_1 \cup S_2} = 6$ when $\abs{S_1} = \abs{S_2} = 3$ and
the two triples do not intersect. 

Let $K := S_1 \cup S_2$ represent each of the isomorphic configurations.
While the first moment has only 2 isomorphic configurations
$\{x,x,2x\}$ and $\{x,y,x+y\}, \;\ x \neq y$, the second moment has 42 configurations. 

We denote the {\it weight} $W(K)$ to be the number of 
isomorphic configuration $K$ that occurs in the sum.
For each $K$, we find $E[X_{S_1}X_{S_2}] $ and $W(K)$,
then apply the following relation for $E[X^2],$
\[
E[X^2] =  \sum_{[S_1,S_2]}E[X_{S_1}X_{S_2}] 
=\sum_{K \in \mathcal{K}} E[X_{S_1}X_{S_2}] \cdot W(K).
\] 

The first quantity, $E[X_{S_1}X_{S_2}]$ is not hard to compute. 
Let $p := \abs{S_1 \cup S_2}$.
\[ E[X_{S_1}X_{S_2}] = 
\begin{cases} \dfrac{c}{c^p}, & \text{ if } \abs{S_1 \cap S_2} \neq 0, \\
 \dfrac{c^2}{c^p} & \text{ if } \abs{S_1 \cap S_2} = 0.
\end{cases}\]

The second quantity, $W(K)$ is harder to compute. One way to do so is to use
Goulden-Jackson Cluster method, see \cite{GouldJack}.
This method gives us a generating function in terms of $n$. 

\textbf{Formulas:} $E[X^2]$ has 12 formulas up to the values of $n$ mod 12.

We obtain the formulas by applying both Goulden-Jackson Cluster method
and {\it polynomial ansatz}.
For each $K$ contributes to $E[X^2]$, we first find the generating function using
Goulden-Jackson Cluster method. This gives us the period $l$ of $K$.
Then we count $W(K)$ numerically
for each $n$. Finally we interpolate the polynomial, $f(n)$
of degree at most 4 with each period $l$ that fits the numerical results.
This polynomial ansatz actually gives the rigorous proof of the second moment. 

As an example, we show the formula of $E[X^2]$ for $n \equiv$ 1 mod 12,
\[E[X^2] = \dfrac{(n-1)(24c^3-76c-27n+65-9n^2+12cn^2+24c^2n+3n^3-16c^2)}{48c^4}.\]

The complete solutions of $E[X^2]$ and $E[(X-\mu)^2]$
can be found in Maple program
\texttt{SchurT.txt} on the author's web site. 

\textbf{The higher moment:} 

We now consider $m$-tuples $[S_1,S_2,...,S_m]$ of triples $\{x,y,x+y\}$ of $[1,n].$ 
\[E[X^m] = \sum_{[S_1,S_2,..,S_m]}E[X_{S_1}X_{S_2}..X_{S_m}].\]

We again consider for each isomorphic configuration $K$,
its weight $W(K)$ and its probability $E[X_{S_1}X_{S_2}..X_{S_m}]$. 

For $E[X^3]$, there are more than 500 isomorphic configurations
of $S_1 \cup S_2 \cup S_3$
with at least 72 formulas up to the values of $n$ mod 72. 
We tried for about 3 weeks, but in the end, 
we felt like it is too much. At least 
we know it is {\it not easy} to compute these moments.

\section{Method of Moments for Asymptotic Normality}

In the case where we work with objects with a nice generating,
we will try to show that asymptotically their distribution is normal.
In this section, we will layout the method of moments for this purpose.
We first define the moment generating function $\phi(t).$

\begin{defi}
$\phi(t) := E[e^{tX}]= \begin{cases} \sum\limits_x e^{tx}p(x) &\mbox{if } x \mbox{ is discrete} \\ 
\int_{-\infty}^{\infty} e^{tx}f(x) dx& \mbox{if } x \mbox{ is continuous}. \end{cases}$ \\
\end{defi}

Moment generating function of standard normal distribution looks like 
\[ \phi(t) = \dfrac{1}{\sqrt{2\pi}}\int_{-\infty}^{\infty} e^{tx}e^{-\frac{x^2}{2}} dx 
 = \dfrac{1}{\sqrt{2\pi}}e^{\frac{t^2}{2}}\int_{-\infty}^{\infty} e^{-\frac{(x-t)^2}{2}} dx  
 = e^{\frac{t^2}{2}}. \]

To show asymptotic normality of some combinatorial objects,
we will show that, asymptotically, their moment generating functions 
(after centralize and normalize) are $e^{\frac{t^2}{2}}.$
 
Another method is to match their coefficients. 
The Maclaurin series of 
$e^{\frac{t^2}{2}}$ is $\displaystyle \sum_{r=0}^{\infty} \dfrac{t^{2r}}{r!2^r}.$ 

On the other hand, the general series expansion of moment
generating function (for the discrete objects) is
\[ \phi(t)  = E[e^{tX}]  = \sum_x\sum_n \dfrac{t^nx^n}{n!}p(x)
= \sum_n \dfrac{t^n}{n!} \sum_x x^np(x) = \sum_n \dfrac{t^n}{n!}m_n,    \]
where $m_n = E[X^n]$. We note that $E[X^n] = \phi^{n}(0)$ as well.

Comparing these two series, we need to show that 
$m_{2r}=\dfrac{(2r)!}{r!2^r}$ and $m_{2r-1}=0.$

\section{Inversion and Major Index}
Inversion number of a permutation is the number of ``misplace''
between the pair of elements. Let $\pi$  be a permutation. 
The inversion number, $inv(\pi),$ 
is the number of pair $(i,j)$ such that $i>j$ and $\pi(i)<\pi(j).$
For example, $inv(52314) = 6$ from the pairs 
$(5,2), (5,3), (5,1), (5,4), (2,1), (3,1).$ As an application, the 
inversion numbers are used in the 
Laplace expansion when we calculate the determinant
of squared matrix.

The probability generating function over 
all permutation of length $n$ is defined as 
\[  F_n(q) := \dfrac{1}{n!}\sum_{\pi \in S_n} q^{inv(\pi)}  .\]

$F_n(q)$ has a nice formula which can be shown by induction on $n$,
\[    F_n(q) = \dfrac{1}{n!}\prod_{i=1}^n \dfrac{1-q^i}{1-q} .\]
The mean, $\mu_n := E[X]$, of inversion number over all
permutation of length $n$ is $\dfrac{n(n-1)}{4}$.  
This can be seen combinatorially or through generating function,
i.e. $E[X] = F_n'(q)|_{q=1}.$
Therefore, the probability generating function about the mean is
\[   G_n(q) = \dfrac{1}{n!q^{n(n-1)/4}}\prod_{i=1}^n \dfrac{1-q^i}{1-q} .\] 
From here, we (or our computer friend) derive that 
the variance, $\sigma_n^2 = Var_n := E[(X-\mu)^2]$, is $\dfrac{n(n-1)(2n+5)}{72}.$  

On the other hand, the major index of a permutation, $maj(\pi),$ 
is the sum of the positions of the descents of the permutation, i.e.
\[ maj(\pi) =  \sum_{\pi(i)>\pi(i+1)}  i. \]
For example  $maj(52314) = 1+3 = 4.$

It is a well known result that the distributions over all 
permutations of length $n$ of inversion number
and major index are the same. This fact was shown
by MacMahon in 1913, \cite{Mac}. 

William Feller (\cite{Fe}, 3rd ed., p.257) 
proved that the number of inversions (and hence also the major
index) is asymptotically normal. 
That is if we let $Y_n$ be the centralized 
and normalized random variable $Y_n = \dfrac{inv - \mu_n}{\sigma_n},$
then $Y_n \to N$, as $n \to \infty$, 
in distribution, where $N$ is the Gaussian distribution whose probability
density function is $\dfrac{e^{\frac{-x^2}{2}}}{\sqrt{2\pi}}$.

Moreover, in 2010, Baxter and Zeilberger, \cite{Z6}, showed that
the inversion numbers and the major index are
asymptotically \textbf{joint-independently-normal.} 
This is an impressive result. 

In this section, we show the asymptotic normal 
distribution of inversion numbers
(also number of major index). We follow the
method in \cite{Z5}. 
The outline of the method is similar to Feller. 
We show that 
 $ Y_n := \dfrac{X_n-n(n-1)/4}{\sqrt{n(n-1)(2n+5)/72}} \sim N(0,1)$ as $n \to \infty$.
Although this time we show by moment calculation method, i.e.
$m_{2r}=\dfrac{(2r)!}{2^rr!}$ and $m_{2r-1}=0$ 
for every $r$, where $m_r = E[Y^r]$.
Note that the odd moment case is obviously seen from symmetry.

\subsection{Asymptotic-Normality of Inversion Numbers }  \label{inv3}
The goal is to find the binomial moment $B_r(n) := E[\binom{X_n-\mu}{r}]$ 
(before we convert it back to $m_r$).
This binomial moment is the coefficient of Taylor series expansion of $G_n(q)$ at 1.
Writing $q=1+z$ then
\begin{align*}  
G_n(1+z) &= \sum_{k=0}^{\infty}p(x=k)(1+z)^k
= \sum_{k=0}^{\infty}p(x=k) \sum_{r=0}^k \binom{k}{r} z^r \\
&=  \sum_{r=0}^{\infty}  z^r  \sum_k \binom{k}{r}  p(x=k)
 = \sum_{r=0}^{\infty} B_r(n)z^r.   
 \end{align*}

Toward that, we consider
\[  \dfrac{G_n(q)}{G_{n-1}(q)} = \dfrac{1}{nq^{(n-1)/2}}\cdot \dfrac{1-q^n}{1-q} .\]

We then define $P(n,z)$ by replacing $q$ with $1+z$ in the above equation:
\[   P(n,z) = \dfrac{1}{n(1+z)^{(n-1)/2}}\cdot \dfrac{1-(1+z)^n}{1-(1+z)} . \]

The Taylor expansion of $P(n,z)$ from Maple looks like
\begin{align*}
P(n,z) &= 1+\frac{(n-1)(n+1)}{24}z^2-\frac{(n-1)(n+1)}{24}z^3 \\
& +\frac{(n-1)(n+1)(n^2+71)}{1920}z^4-\frac{(n-1)(n+1)(n^2+31)}{960}z^5+\dots
\end{align*}

The general formula is not difficult to find:
\[  P(n,z) = \sum_{i=0}^{\infty} p_i(n) z^i, \]
where
\[  p_i(n) = \frac{1}{n}\sum_{s=0}^i (-1)^s
\binom{\frac{n-3}{2}+s}{s}\binom{n}{i+1-s}.  \]

Now consider the recurrence
\[  G_n(1+z) = P(n,z)G_{n-1}(1+z) . \]

By comparing the coefficient of $z^r$ 
on both sides of the equation, we have
\begin{equation} \label{rec4}
B_r(n)-B_r(n-1) = \sum_{s=1}^r p_s(n)B_{r-s}(n-1). 
\end{equation}

With this recurrence we can prove the conjecture 
of leading term of $B_r(n)$ by an induction on $r$. 

We see from the formula that the leading term of $p_i(n)$ is
\[ \mathcal{L}\{p_i(n)\} =  \begin{cases} \displaystyle
\sum_{s=0}^i \dfrac{(-1)^s}{2^ss!(i+1-s)!}n^i  = \dfrac{n^i}{2^i(i+1)!},  & \mbox{ when $i$ is even,}  \\
-\dfrac{i-1}{2^i \cdot i!}n^{i-1},        & \mbox{ when $i$ is odd.}
\end{cases}   \]

We now can prove the leading term of $B_R(n).$
\begin{thm}
\[B_{2r}(n) = \dfrac{n^{3r}}{r!2^{3r}3^{2r}} + \;\ \mbox{ lower order terms ,}\]
\[B_{2r+1}(n) = -\dfrac{n^{3r}}{(r-1)!2^{3r}3^{2r}} + \;\ \mbox{ lower order terms.}\]
\end{thm}

\begin{proof} We calculate directly  that $B_0(n)=1$ and $B_1(n)=0.$ 
For $R \geq 2$,
we verify the leading terms on both sides of \eqref{rec4}
which in turn equivalent to do induction on $R.$ 

Case  $R=2r$: 

The leading term of the left hand side of \eqref{rec4}:
\[  \mathcal{L}\{ B_{2r}(n)-B_{2r}(n-1)\} = \dfrac{3r}{r!2^{3r}3^{2r}}n^{3r-1}.  \]

The leading term of the right hand side of \eqref{rec4}:
\begin{align*}
 \mathcal{L}\{  \sum_{s=1}^{2r} p_s(n)B_{2r-s}(n-1)\} 
&=  \mathcal{L}\{  p_2(n)B_{2r-2}(n-1)\}  \\
&= \dfrac{n^2}{24}\cdot \dfrac{n^{3r-3}8\cdot9}{(r-1)!2^{3r}3^{2r}} \\
&= \dfrac{n^{3r-1}3}{(r-1)!2^{3r}3^{2r}}. 
\end{align*}

Case  $R = 2r+1$: \\ \\
The leading term of the left hand side of \eqref{rec4}:
\[  \mathcal{L}\{ B_{2r+1}(n)-B_{2r+1}(n-1)\} = -\dfrac{3r}{(r-1)!2^{3r}3^{2r}}n^{3r-1}.  \]

The leading term of the right hand side of \eqref{rec4}:
\begin{align*}
 \mathcal{L}\{  \sum_{s=1}^{2r+1} p_s(n)B_{2r+1-s}(n-1)\} 
&=  \mathcal{L}\{  p_1(n)B_{2r}(n-1) + p_2(n)B_{2r-1}(n-1) + p_3(n)B_{2r-2}(n-1)\}  \\
&= 0-\dfrac{n^2}{24}\cdot \dfrac{n^{3r-3}8\cdot9}{(r-2)!2^{3r}3^{2r}} 
-\dfrac{2n^2}{8\cdot 3!}\cdot \dfrac{n^{3r-3}8\cdot9}{(r-1)!2^{3r}3^{2r}}\\
&= -\dfrac{n^{3r-1}}{(r-1)!2^{3r}3^{2r}}(3(r-1)+3).
\end{align*}
The leading terms on both sides are equal in both cases.
\end{proof}

\begin{rem}
In fact we can show as many leading terms as we'd like i.e.
\[B_{2r}(n) = \dfrac{n^{3r}}{r!2^{3r}3^{2r}}\left[1 
+ \dfrac{3r(31-6r)}{50n} +\mathcal{O}(1/n^2) \right],\]
\[B_{2r+1}(n) = -\dfrac{n^{3r}}{(r-1)!2^{3r}3^{2r}} 
\left[1 + \dfrac{3r(31-6r)}{50n} +\mathcal{O}(1/n^2) \right].\]
We encourage the reader to check the calculation themselves. 
\end{rem}

The last step is to turns the binomial moments $B_r(n)$ 
to the straight moment $M_r(n)$ using the well known identity:
\[   M_r(n) = \sum_{i=0}^r {r \brace i} B_i(n) \cdot i!,   \]
where ${r\brace i}$ is a Stirling number of the second kind
i.e. number of ways to partition $r$ objects 
into $i$ non-empty subsets.

Then the leading term of $M_r(n)$ is
\[   \mathcal{L}\{ M_{2r}(n) \} = \mathcal{L}\{B_{2r}(n) \}\cdot (2r)!,     \]
and
\[   \mathcal{L}\{ M_{2r+1}(n) \} = \mathcal{L}\{B_{2r+1}(n) \}\cdot (2r+1)!
+ \binom{2r+1}{2} \mathcal{L}\{B_{2r}(n)\}\cdot(2r)!.     \]

The simple calculation leads us to next corollary. 
\begin{cor}
\begin{align*}
M_{2r}(n) &= \dfrac{(2r)!}{r!2^{3r}3^{2r}}n^{3r} + \;\ \mbox{ lower order terms ,}\\
M_{2r+1}(n) &=  0 \cdot n^{3r} + \;\ \mbox{ lower order terms.}
\end{align*}
\end{cor}
The coefficient of $n^{3r}$ in $M_{2r+1}(n)$ is 0. 
So we know that the actual leading term has 
degree in $n$ less than $3r.$

It is now easy to verify the claim that 
$m_{2r} := \dfrac{M_{2r}(n)}{M_2(n)^r}$
is $\dfrac{(2r)!}{2^rr!}$ and $m_{2r+1} := \dfrac{M_{2r+1}(n)}{M_2(n)^{r+1/2}}$
is $0$  asymptotically.

\begin{rem} It is possible
to extend the result to as many leading terms of $m_{r}$ as we'd like,  i.e.
\begin{align*}
m_{2r} &= \dfrac{(2r)!}{2^rr!}\big(  
1-\dfrac{9r(r-1)}{25n}+\dfrac{9r(r-1)(441r^2-205r-129)}{61250n^2} \\
&-\dfrac{3r(r-1)(7938r^4-3132r^3-27252r^2-20202r+2221835)}{3062500n^3}+...\big), \\
m_{2r+1} &= 0. 
\end{align*}
\end{rem}

\subsection{Asymptotic-Normality of Major Index}

Asymptotic-Normality of Major Index already follows 
from those of inversion numbers. However this time we 
want to show it by starting with recursive relations right away 
rather than the known generating function. 
This recursive relation method is from the 
awesome paper \cite{Z6} mentioned earlier. 

\textbf{Generating Function}

Define generating function of major index:
\[  H_n(q) := \sum_{\pi \in S_n} q^{maj(\pi)}.  \]

The recurrence relation comes from considering
the last 2 elements in the permutation.
\[  maj(\pi''ji) = \begin{cases}
maj(\pi''j)    & \mbox{if } j < i; \\
maj(\pi''j) +n-1   & \mbox{if } j > i.
\end{cases} \]

So in order to compute $H_n(q)$, we introduce the
more general counting function of permutations (in $S_n$)
that end with an $i.$ Let's call it $F(n,i)(q)$.
\[  F(n,i)(q) := \sum_{\substack{\pi \in S_n \\ \pi(n)=i}} q^{maj(\pi)}.  \]
It follows that (dropping the argument $q$):
\[  F(n,i) = \sum_{j=1}^{i-1}F(n-1,j)+q^{n-1}\sum_{j=i}^{n-1}F(n-1,j),  \]
with $F(1,1)=1.$ Note that, after chopping $i$, 
the index of the second sum has been rearranged.

Replacing $i$ by $i+1$ in the above equation, we have:
\[  F(n,i+1) = \sum_{j=1}^{i}F(n-1,j)+q^{n-1}\sum_{j=i+1}^{n-1}F(n-1,j).  \]

Then subtract the former with the latter to get
\begin{align} \label{F1}
F(n,i) - F(n,i+1) &= -F(n-1,i)+q^{n-1}F(n-1,i) \nonumber \\
&= -(1-q^{n-1})F(n-1,i)     \;\ \;\ \;\ \mbox{ for } 1\leq i < n.
\end{align}
Another important relation,
\begin{equation}  \label{F2}
F(n,n) = \sum_{j=1}^{n-1} F(n-1,j).
\end{equation}

Finally, the connection between $H_n$ and $F(n,i)$ is 
\[H_n(q) = F(n+1,n+1)(q).\]

\textbf{Centralized Probability Generating Function}  

Consider the centralized probability generating function:
\[  G(n,i)(q) := \dfrac{F(n,i)(q)}{(n-1)!q^{n-i+(n-1)(n-2)/4}}. \]

The recurrences \eqref{F1} and \eqref{F2} become:
\begin{equation} \label{G1}
G(n,i) = \frac{1}{q}G(n,i+1) +\frac{q^{n-1}-1}{q^{n/2}(n-1)}G(n-1,i)   
 \;\ \;\ \;\ \mbox{ for } 1\leq i < n,
\end{equation}
and
\begin{equation}  \label{G2}
G(n,n) = \frac{1}{n-1}\sum_{j=1}^{n-1} q^{n/2-j}G(n-1,j).  
\end{equation}

\textbf{Guessing the Binomial Moments}

The equations \eqref{G1} and \eqref{G2} help us to
crank out $G(n,i)(q)$ very quickly which will help us
to conjecture the binomial moment $B_r(n,i)$ through
the Taylor series:
\[   G(n,i)(1+z) = \sum_{r=0}^{\infty}B_r(n,i)z^r.   \] 

The conjectures that we made here are similar to what we had before, i.e.
for $1 \leq i \leq n,$
\[  B_{2r}(n,i) = \dfrac{1}{2^rr!}\left(\dfrac{n^3}{36}\right)^r +
\text{ lower-order-term-in-}n. \]
\[  B_{2r+1}(n,i) = -\dfrac{r}{2^rr!}\left(\dfrac{n^3}{36}\right)^r +
\text{ lower-order-term-in-}n. \]

\textbf{Guessing is nice and all. But how to prove?}

You can make a recurrence out of \eqref{G1} and \eqref{G2}, i.e. 
comparing coefficient of $z^r$ on both sides:
\begin{align}  
  B_{r}(n,i) &=  B_{r}(n,i+1) -B_{r-1}(n,i+1) +B_{r-1}(n-1,i) 
   +B_{r-2}(n,i+1)   \nonumber \\
 & -B_{r-2}(n-1,i) -B_{r-3}(n,i+1) 
 +\left(\frac{n^2}{24} -\frac{n}{12}+1 \right)B_{r-3}(n-1,i)  \nonumber \\  
 & +B_{r-4}(n,i+1)
 +\left(\frac{-n^2}{12} +\frac{n}{6}-1 \right)B_{r-4}(n-1,i)  \nonumber \\  
 &+\text{ lower-order-term-in-}n, \;\ \;\  1\leq i < n ,\\
 B_r(n,n) &= \dfrac{1}{n-1}  \sum_{j=1}^{n-1} \big[  B_r(n-1,j) 
+   \left(\frac{n}{2}-j\right) B_{r-1}(n-1,j)  \nonumber \\
&+  \binom{\dfrac{n}{2}-j}{2} B_{r-2}(n-1,j)
 + \binom{\dfrac{n}{2}-j}{3} B_{r-3}(n-1,j)    \nonumber \\
& +\text{ lower-order-term-in-}n  \big].
\end{align}

It is possible to show that $G(n,i)(q)$ are the same for all $1 \leq i \leq n$.
As a result $B_r(n,i)=B_r(n,j)$ for all $1 \leq i,j \leq n$.
We simplify (6) and (7) using this claim. 
\begin{align}  
 B_{r}(n,i) - B_{r}(n-1,i) &=  
   B_{r-1}(n,i) -B_{r-1}(n-1,i)  \nonumber \\
 &  -B_{r-2}(n,i) 
 +\left(\frac{n^2}{24} -\frac{n}{12}+1 \right)B_{r-2}(n-1,i)  \nonumber \\  
 & +B_{r-3}(n,i)
 +\left(\frac{-n^2}{12} +\frac{n}{6}-1 \right)B_{r-3}(n-1,i)  \nonumber \\  
 &+\text{ lower-order-term-in-}n, \;\ \;\  1\leq i < n ,\\
 B_r(n,n) - B_r(n-1,j)  &=  \dfrac{n(n-2)}{24}B_{r-2}(n-1,j) 
 - \dfrac{n(n-2)}{24}B_{r-3}(n-1,j) \nonumber \\
&  +\text{ lower-order-term-in-}n  .
\end{align}

We note that  $\displaystyle \sum_{j=1}^{n-1} \left(\frac{n}{2}-j\right) = 0,$
 $\displaystyle \sum_{j=1}^{n-1} \binom{\frac{n}{2}-j}{2}  = \dfrac{n(n-1)(n-2)}{24}$
 and  $\displaystyle \sum_{j=1}^{n-1} \binom{\frac{n}{2}-j}{3}  = -\dfrac{n(n-1)(n-2)}{24}.$

We could prove the conjectures by 
applying either (8) or (9). 
We choose to show by (8).

\begin{thm} \label{Three} For $1 \leq i \leq n,$
\[  B_{2r}(n,i) = \dfrac{1}{2^rr!}\left(\dfrac{n^3}{36}\right)^r +
\text{ lower-order-term-in-}n. \]
\[  B_{2r+1}(n,i) = -\dfrac{r}{2^rr!}\left(\dfrac{n^3}{36}\right)^r +
\text{ lower-order-term-in-}n. \]
\end{thm}

\begin{proof} We calculate directly  that $B_0(n,i)=1$ and $B_1(n,i)=0.$ 
For $R \geq 2$,
we verify the leading terms on both sides of  (8)
which in turn equivalent to do induction on $R.$ 

Case  $R=2r$: 
The leading term of the left hand side of (8):
\[  \mathcal{L}\{ B_{2r}(n,i) - B_{2r}(n-1,i)\} = \dfrac{3r}{r!72^{r}}n^{3r-1}.  \]

The leading term of the right hand side of (8):
\begin{align*}
 \mathcal{L}\{ \dfrac{n^2}{24}B_{2r-2}(n-1,i) \} 
&= \dfrac{n^2}{24}\cdot \dfrac{n^{3r-3}}{(r-1)!72^{r-1}} 
= \dfrac{3r \cdot n^{3r-1}}{r!72^{r}}. 
\end{align*}

Case  $R=2r+1$: 
The leading term of the left hand side of (8):
\[  \mathcal{L}\{ B_{2r+1}(n,i) - B_{2r+1}(n-1,i)\} = \dfrac{-3r^2}{r!72^{r}}n^{3r-1}.  \]

The leading term of the right hand side of (8):
\begin{align*}
 \mathcal{L}\{ RHS \} 
&=  \dfrac{3r\cdot n^{3r-1}}{r!72^r}
-\dfrac{n^2}{24}\cdot \dfrac{(r-1)n^{3r-3}}{(r-1)!72^{r-1}} 
-\dfrac{n^2}{12}\cdot \dfrac{n^{3r-3}}{(r-1)!72^{r-1}} 
= \dfrac{-3r^2 \cdot n^{3r-1}}{r!72^{r}}. 
\end{align*}
\end{proof}

Lastly we define $\displaystyle G_n(1+z) := \sum_{i=1}^n G(n,i)(1+z) = G(n+1,n+1)(1+z).$
The leading term from theorem \ref{Three} implies 
the leading term of $B_r(n)$ of 
$G_n(1+z) = \sum_{r=0}^{\infty}B_r(n)z^n$ since
we can see that $B_r(n) = B_r(n+1,n+1).$
From this point on, the normality distribution of 
the random variable of major index
on permutation of length $n$
can be shown the same way as in the last part of 
section \ref{inv3} on inversion numbers.

\section {Boolean Functions}

\textbf{Boolean function} is function that
sends each element in the \textit{Boolean domain} 
$B^n := \{0,1\}^n$ to $\{0,1\}.$

For example, a boolean function with $n=3$ is 
shown in the form of the truth table, 
\begin{center}
\begin{tabular}{ l |c } 
  000 & 0  \\
  001 & 1  \\
  010 & 0  \\
  011 & 1  \\
  100 & 1  \\
  101 & 1  \\
  110 & 0  \\
  111 & 1  \\
\end{tabular}
\end{center}

The boolean function in a ``disjunctive normal form'' (DNF)  
that represents this table is
\[   \bar{x}_1\bar{x}_2x_3 \lor \bar{x}_1x_2x_3 \lor x_1\bar{x}_2\bar{x}_3
\lor x_1\bar{x}_2x_3 \lor x_1x_2x_3.\]

With the length $n$ is 3, 
the possible number of terms is $2^3=8$
and the possible number of boolean functions is $2^{2^3} = 2^8 = 256.$
In general, there are $2^{2^n}$ possibilities of boolean function of length $n$. 

We write the boolean function by the set of elements that map to 1.
In the example above, $f =\{001,011,100,101,111\}$ .
From this point of view, the boolean function is a set of 
vertices in a $n$-dimensional unit cube. 
From this point onward, we will look at the boolean 
function from this angle and denote the set of 
vertices as $f$.

\subsection{Asymptotic Normality of Number of Elements in $f$}

The result from the name of this section might sounds obvious.
The distribution is just a binomial. But we will calculate the moments
and hopefully obtain some formulas and identities along the way.

Let $\mathcal{B}$ be a sample space of all $2^{2^n}$
boolean functions.  Let $X := X_0(f)$ be 
a random number of 0-cube (0-dimension subspace), 
i.e. number of elements that are contained in the set $f$ 
of boolean function.
Of course,  $0 \leq X \leq 2^n.$ We want to work toward the 
asymptotic normality of $X.$ Basically, we want to show that
 $ Y_n := \dfrac{X_n-\mu}{\sigma} \sim N(0,1)$ as $n \to \infty$.
We actually show the even moment of $Y_n: \;\ m_{2r}=\dfrac{(2r)!}{2^rr!}$ 
and the odd moment of $Y_n: \;\ m_{2r-1}=0$ for every $r$.
But the odd moment case is obviously seen from symmetry.
Therefore we only show that asymptotically 
$\dfrac{E[(X-\mu)^{2r}]}{\sigma^{2r}} = \dfrac{(2r)!}{2^rr!}. $

\textbf{The Generating Function Method}

Define the generating function $F_n(q)$ by $F_n(q) = \sum_{k=0}^{\infty}  a_k  q^k,  $
where $k$ is the number of elements in the set $f$ and 
$a_k$ is the number of ways to have $k$ elements with each
elements have length $n.$ The formula for $F_n(q)$ is 
(each element has 1/2 chance to be in $f$ and 1/2 chance to be out of $f$),
\[ F_n(q) = \left(\dfrac{1+q}{2}\right)^{2^n}.\]

The generating function about the mean $G_n(q)$ is
\[ G_n(q) =   \left(\dfrac{1+q}{2}\right)^{2^n}\dfrac{1}{q^{2^{n-1}}}. \]
This binomial moment is the coefficient of Taylor series expansion of $G_n(q)$ at 1.
Writing $q=1+z$, we have
\[  G_n(1+z) = \sum_{r=0}^{\infty} B_r(n)z^r.   \]
Toward that, we consider
\[  \dfrac{G_n(q)}{G_{n-1}(q)} =  \left[\dfrac{1+q}{2\sqrt{q}}\right]^{2^{n-1}} .\]

We then define $P(n,z)$ by replacing $q$ with $1+z$ in the above equation:
\begin{equation} \label{MyP}
  P(n,z)  := \dfrac{G_n(1+z)}{G_{n-1}(1+z)}  =  \left[\dfrac{2+z}{2\sqrt{1+z}}\right]^{2^{n-1}} . 
\end{equation}

The Taylor expansion of $P(n,z)$ from Maple looks like
\begin{align*}
P(n,z) &= 1+\frac{2^n}{16}z^2-\frac{2^n}{16}z^3 
+ \left(\frac{4^n}{512}+\frac{7 \cdot 2^n}{128} \right)z^4 \\
& -\left(\frac{4^n}{256}+\frac{3 \cdot 2^n}{64} \right)z^5+\dots
\end{align*}

Let the expansion of $P(n,z)$ be
\[  P(n,z) = \sum_{i=0}^{\infty} p_i(n) z^i.\]
It can be shown from \eqref{MyP} that the leading term of $p_i(n)$ is
\begin{align*}
\mathcal{L}\{p_{2r}(n)\} &= \dfrac{2^{rn}}{2^{4r}r!}, \\
\mathcal{L}\{p_{2r+1}(n)\} &= \dfrac{-2^{rn}}{2^{4r}(r-1)!}.
\end{align*}

Now consider the recurrence
\begin{equation} \label{eqG}
  G_n(1+z) = P(n,z)G_{n-1}(1+z) . 
\end{equation}

By comparing the coefficient of $z^r$ in \eqref{eqG} 
on both sides of the equation, we have
\begin{equation} \label{rec}
B_r(n) = \sum_{s=0}^r p_s(n)B_{r-s}(n-1). 
\end{equation}

With this recurrence, 
we can prove the leading term of $B_R(n)$, 
by induction on $R$.
\begin{thm}
\[B_{2r}(n) = \dfrac{2^{rn}}{r! 8^r} + \;\ \mbox{ lower order terms in $2^n$,}\]
\[B_{2r+1}(n) = \dfrac{-2^{rn}}{(r-1)!8^r} + \;\ \mbox{ lower order terms in $2^n$.}\]
\end{thm}

\begin{proof} We calculate directly  that $B_0(n)=1$ and $B_1(n)=0.$ 
For $R \geq 2$, we verify the leading terms on 
the right hand sides of \eqref{rec}. 

Case  $R=2r$: 
\begin{align*}
 \mathcal{L}\{  \sum_{s=0}^{2r} p_s(n)B_{2r-s}(n-1)\} 
&=  \mathcal{L}\{   \sum_{t=0}^r p_{2t}(n)B_{2(r-t)}(n-1) \}  \\
&= \sum_{t=0}^r  \dfrac{2^{tn}}{16^tt!} \cdot \dfrac{2^{(r-t)(n-1)}}{(r-t)!8^{r-t}} \\
&= \dfrac{2^{rn}}{16^r}  \sum_{t=0}^r  \dfrac{1}{t!(r-t)!} 
= \dfrac{2^{rn}}{16^r}\dfrac{2^r}{r!} . 
\end{align*}

Case  $R = 2r+1$: 
\begin{align*}
 \mathcal{L}\{  \sum_{s=0}^{2r+1} p_s(n)B_{2r+1-s}(n-1)\} 
&=  \mathcal{L}\{   \sum_{t=0}^r p_{2t}(n)B_{2(r-t)+1}(n-1) 
+  \sum_{t=0}^r p_{2t+1}(n)B_{2(r-t)}(n-1)   \}  \\
&= \sum_{t=0}^r  \dfrac{2^{tn}}{16^tt!} \cdot \dfrac{-2^{(r-t)(n-1)}}{(r-1-t)!8^{r-t}} 
+\sum_{t=0}^r  \dfrac{-2^{tn}}{16^t(t-1)!} \cdot \dfrac{2^{(r-t)(n-1)}}{(r-t)!8^{r-t}} \\
&= \dfrac{-2^{rn}}{16^r}  \sum_{t=0}^r  
\left(  \dfrac{1}{t!(r-1-t)!}+  \dfrac{1}{(t-1)!(r-t)!} \right)  = \dfrac{-2^{rn}}{16^r}\dfrac{2^r}{(r-1)!}.
\end{align*}
\end{proof}

Once these have been verified, we turns the binomial moments $B_r(n)$ 
to the straight moment $M_r(n):= E[(X-\mu)^r]$, similar to section 4.1, 
using the well known identity:
\[   M_r(n) = \sum_{i=0}^r {r \brace i} B_i(n) \cdot i!,   \]

The simple calculation leads us to next corollary. 
\begin{cor} \label{Five}
\begin{align*}
M_{2r}(n) &= \dfrac{(2r)!2^{rn}}{r!8^r} + \;\ \mbox{ lower order terms in $2^n$,}\\
M_{2r+1}(n) &=  C(r) \cdot 2^{rn} + \;\ \mbox{ lower order terms in $2^n$,}
\end{align*}
\end{cor}

It is now easy to verify the condition for normality that 
$m_{2r} = \dfrac{M_{2r}(n)}{M_2(n)^r} = 
\dfrac{(2r)!}{2^rr!}$ and $m_{2r+1} = \dfrac{M_{2r+1}(n)}{M_2(n)^{r+1/2}}
= 0$  as $n \to \infty$.


\subsubsection{Moment Calculations}
In this subsection, we apply the overlapping stage approach
to calculate $E[X^r]$ and relate it to the moment about 
the mean $E[(X-\mu)^r]$.

First we calculate the straight moment $E[X^r], \;\ r \geq 0.$
The first moment is easy.
\[ E[X] = \dfrac{1}{2}\binom{2^n}{1} = 2^{n-1}.\]
For the second moment,
\[ E[X^2] = \sum_{S_1}E[X_{S_1}^2] + \sum_{S_1 \neq S_2}E[X_{S_1}X_{S_2}]
= \dfrac{1}{2}\binom{2^n}{1}+2!\dfrac{1}{2^2}\binom{2^n}{2} =\dfrac{2^{n}}{4}(2^n+1).\]
For the third moment,
\[ E[X^3] = \dfrac{1}{2}\binom{2^n}{1}+(3)2!\dfrac{1}{2^2}\binom{2^n}{2} 
+3!\dfrac{1}{2^3}\binom{2^n}{3} = 2^{2n-3}(2^n+3).\]
For the fourth moment,
\[ E[X^4] = \dfrac{1}{2}\binom{2^n}{1}+(7)2!\dfrac{1}{2^2}\binom{2^n}{2} 
+(6)3!\dfrac{1}{2^3}\binom{2^n}{3} 
+4!\dfrac{1}{2^4}\binom{2^n}{4}= 2^{n-4}(2^{2n}+5\cdot2^n-2)(2^n+1).\]

These moments are calculated based on how they overlap with each other.
In general, the $r^{th}$ moment is
\begin{equation} \label{point} 
E[X^r] = \sum_{i=0}^r  {r\brace i} \dfrac{i!}{2^i} \binom{2^n}{i} 
=  \sum_{i=0}^r  {r\brace i} \dfrac{(2^n)_i}{2^i}   ,
\end{equation} 
where ${r\brace i}$ is again a Stirling number of the second kind.

We use these straight moment to calculate the moments about the mean,
 \begin{equation} \label{MoMean}
  E[(X-\mu)^r] 
= \sum_{i=0}^r (-1)^{r-i} \binom{r}{i}E[X^i] \mu^{r-i}, \;\ \;\ r \geq 0,
\end{equation}
where $\mu= E[X] = 2^{n-1}.$

Some examples of these moments are
\begin{align*}
Var(X) =E[(X-\mu)^2] &= 2^{n-2}  , \\
E[(X-\mu)^4] &= 2^{n-4}(3\cdot2^n-2) , \\
E[(X-\mu)^6] &=  2^{n-6}(15\cdot2^{2n}-30\cdot2^n+16), \\
E[(X-\mu)^8] &=  2^{n-8}(105\cdot2^{3n}-420\cdot2^{2n}+588\cdot2^n-272), \\
E[(X-\mu)^{2r+1}] &= 0 , \;\ r \geq 0.
\end{align*}

The general formula is
\begin{align*}
E[(X-\mu)^r] &=  \sum_{i=0}^r \binom{r}{i}E[X^i] (-2^{n-1})^{r-i} \\
&=  \sum_{i=0}^r \left(\dfrac{-1}{2}\right)^{r-i}\binom{r}{i} 
\sum_{j=0}^i  {i \brace j}(2^n)_j   \dfrac{\left(2^{n}\right)^{r-i} }{2^j}
\end{align*}
Therefore coefficient of $2^{n(r-t)}$ is
\begin{equation} \label{coeff0}
  \sum_{i=t}^r \left(\dfrac{-1}{2}\right)^{r-i}  \binom{r}{i} 
\sum_{j=i-t}^i \dfrac{1}{2^j}{i \brace j}s(j,i-t).  
\end{equation}
where $s(i,k)$ is a Stirling number of the first kind.

\textbf{New identities}

We relate the results of corollary \ref{Five} and
observation that $M_{2r+1}(n)=0$ to
\eqref{coeff0} to gain some new identities.

For case $r$ is odd and $ 0 \leq t \leq r$
and case $r$ is even and $ 0 \leq t < r/2,$
\[ \sum_{i=t}^r \left(\dfrac{-1}{2}\right)^{r-i}  \binom{r}{i} 
\sum_{j=i-t}^i \dfrac{1}{2^j}{i \brace j}s(j,i-t)  = 0.   \]

For $r=2k$ and $t=k$, we have
\[  \sum_{i=t}^r \left(\dfrac{-1}{2}\right)^{r-i}  \binom{r}{i} 
\sum_{j=i-t}^i \dfrac{1}{2^j}{i \brace j}s(j,i-t)  
= \dfrac{(2k)!}{8^kk!} .   \]

\subsection{Moments of Numbers of 1-dimensional Cube in $f$}
Now we consider $X_1(f)$, a random number of 1-dimension subspace 
that is contained in $f$.
For example, consider $n=3, f=\{ 000,001,101,011,111\}$.
We have $X_0(f) = 5$ and $X_1(f) = 5$ namely $00B, 0B1, B01, 1B1, B11$.
Note that $0 \leq X_1(f) \leq n2^{n-1}.$

Define the generating function $F_n(q)$ by $F_n(q) = \sum_{k=0}^{\infty}  a_k  q^k,  $
where $k$ is the number of 1-dimensional cube in the set $f$ 
of boolean function of length $n$ and 
$a_k$ is the number of boolean functions $f$ 
that has $k$ number of 1-dimensional cubes. 
This time $F_n(q)$ looks complicate. Some of them are
\begin{align*}
F_1(q) &= \dfrac{1}{4}(3+q), \\
F_2(q) &= \dfrac{1}{16}(7+4q+4q^2+q^4) \\
F_3(q) &= \dfrac{1}{256}(35+36q+54q^2+40q^3+30q^4+24q^5+16q^6+12q^7+8q^9+ q^{12}).
\end{align*}
It is very difficult, if not impossible, to find a general formula for $F_n(q).$
However we show a strong empirical evidence in section \ref{Exp} 
that the distribution (case $k=1$) once again converges to normal.

Let $X:=X_1(f)$. 
We now apply the overlapping state approach 
to calculate $E[X^r], r \geq 0.$ A few of them are
\begin{align*}
E[X] &= \dfrac{1}{2^2}n2^{n-1}, \\
E[X^2] &= \mbox{ 1-cube overlap} + \mbox{ 0-cube overlap} + \mbox{ no overlap} \\
&= \sum_{i=1}^{n2^{n-1}}E[X_i^2] + \sum_{i=1}^{4n(n-1)2^{n-2}}E[X_i^2] + \sum_{i=1}^{Rest}E[X_i^2] \\
&= \dfrac{n2^{n-1}}{2^2}+\frac{4n(n-1)2^{n-2}}{2^3}+\dfrac{(n2^{n-1})^2-4n(n-1)2^{n-2}-n2^{n-1}}{2^4} \\
&= \dfrac{n2^n}{64}(2+4n+n2^n) ,   
\end{align*}
\begin{align*}
E[X^3] &= \mbox{ 1-cube overlap} + \mbox{ 0-cube overlap} + \mbox{ no overlap} \\
&= \dfrac{a}{2^2}+\frac{b}{2^3} +\frac{c}{2^4}
+\frac{3a(a-1)}{2^4}+\dfrac{b(a-3n+2)}{2^5}\\
&  +\dfrac{a^3-[a+b+c   +3a(a-1)   +b(a-3n+2)]}{2^6},  \\
& \mbox{where } \;\ a=n2^{n-1}, b=3\cdot 4n(n-1)2^{n-2} 
\mbox{ and } \;\ c= 4\cdot 8n(n-1)(n-2)2^{n-3} \\ 
&= \dfrac{n^2}{2^9}[24n2^n+6(2n+1)2^{2n}+n2^{3n}].
\end{align*}
  
We did not manage to find formulas for other moments.
Some moments about the mean are
\begin{align*}
Var(X) = E[(X-\mu)^2]  &= \dfrac{n(2n+1)2^{n}}{32}  , \\
E[(X-\mu)^3] &= \dfrac{3n^32^n}{64} .
\end{align*}

\subsection{Moments of Numbers of $k$-dimensional Cube in $f$}  \label{BFMo}
 Let $k$ be the size of the cube.
Example, the cube of size 2 is $\{ 00,01,10,11 \}$
or the cube of size 2: 11B0B is \{ 11000, 11001, 11100, 11101\}.
Let the random variable $X_k$ be the number of $k$-dimensional cubes
in the boolean function of length $n$. 

Our goal is to design an algorithm that inputs symbolic $n$ and $k$
and numeric $r$ and outputs $E[X^r_k]$, the $r$-th moment
of $k$-dimensional cubes contains in boolean function $f$, of length $n$ 
($n$ variables). 

The first moment (for any $k$) is
\[ 
E[X] = \sum_{i=1}^{\binom{n}{k}2^{n-k}} E[X_{i}] = \binom{n}{k}2^{n-k} \cdot \frac{1}{2^{2^k}},
\]
since $E[X_{i}] = P(X_i = 1) =  \dfrac{1}{2^{2^k}}$ 
as all the $2^k$ entries of $k$-cube must be there. 

The calculation of the second moment is more complicated as 
we must keep track of interactions between $X_{i}$ and $X_{j}$. 

For $k=2:$
\begin{align*}
E[X^2] &= \mbox{ 2-cube overlap} +\mbox{ 1-cube overlap} +
\mbox{ 0-cube overlap}+ \mbox{ no overlap} \\
&= \binom{n}{2}2^{n-2}E[X_1^2] 
+ 4\binom{n}{1,1,1,n-3}2^{n-3}E[X_2^2] \\
&+ 16\binom{n}{2,2,n-4}2^{n-4}E[X_3^2]
 + Rest \cdot E[X_4^2] \\
&= \dfrac{\binom{n}{2}2^{n-2}}{2^4} 
+ \dfrac{4\binom{n}{1,1,1,n-3}2^{n-3}}{2^6} 
+ \dfrac{16\binom{n}{2,2,n-4}2^{n-4}}{2^7}
 + \dfrac{ Rest}{2^8} \\
&= \dfrac{n(n-1)}{2^{14}}[8(2n^2+2n+3)2^n+n(n-1)4^n]  .    
\end{align*}

We can see the pattern for general $k$:  (think of $i$ as an $i$-cube intersects
between two objects) 
\begin{align*}
E[X^2] &= \sum_{i=0}^{k} 2^{2(k-i)}\binom{n}{i,k-i,k-i,n-2k+i}2^{n-2k+i} \cdot \dfrac{1}{2^{2^{k+1}-2^{i}}}
+ \dfrac{Rest}{2^{2^{k+1}}} \\
&= \sum_{i=0}^{k} \dfrac{\binom{n}{i,k-i,k-i,n-2k+i}2^{n-i} }{2^{2^{k+1}-2^i}}
+ \dfrac{Rest}{2^{2^{k+1}}} \\
&= \sum_{i=0}^{k} \dfrac{\binom{n}{i,k-i,k-i,n-2k+i}2^{n-i} }{2^{2^{k+1}}}\cdot(2^{2^i}-1)
+ \dfrac{\left[\binom{n}{k}2^{n-k}\right]^2}{2^{2^{k+1}}}   ,
\end{align*}
where $\displaystyle Rest =  \left[\binom{n}{k}2^{n-k}\right]^2
- \sum_{i=0}^{k} \binom{n}{i,k-i,k-i,n-2k+i}2^{n-i} .$ 

The formula for the case $k=3$ is
\[ E[X^2] = \dfrac{n(n-1)(n-2)}{2^{24}3^2}[16(4n^3+6n^2+80n+363)2^n+n(n-1)(n-2)4^n].\]

The general formula for variance is
\[ E[(X-\mu)^2] = E[X^2]-E[X]^2 = \sum_{i=0}^{k} 
\dfrac{\binom{n}{i,k-i,k-i,n-2k+i}2^{n-i} }{2^{2^{k+1}}}\cdot(2^{2^i}-1). \]
The leading term of $E[(X-\mu)^2]$ can be seen to be $\dfrac{n^{2k}2^n}{(k!)^22^{2^{k+1}}}$ .

The calculation for the third moment
is much more complicated,
we only manage to find it for 
the case $k=0$ and $k=1$.


\section{Dominos on Boolean Boards}

Consider the rectangle board of size $m \times n$ with
each of the square filled with number 0 or 1 randomly.
Let $X$ be the random variable of number of 
2-by-1 domino piece with ``the same number''   
vertically or horizontally. 

For example, on 3-by-3 board:
\begin{center}
\includegraphics[scale=0.1]{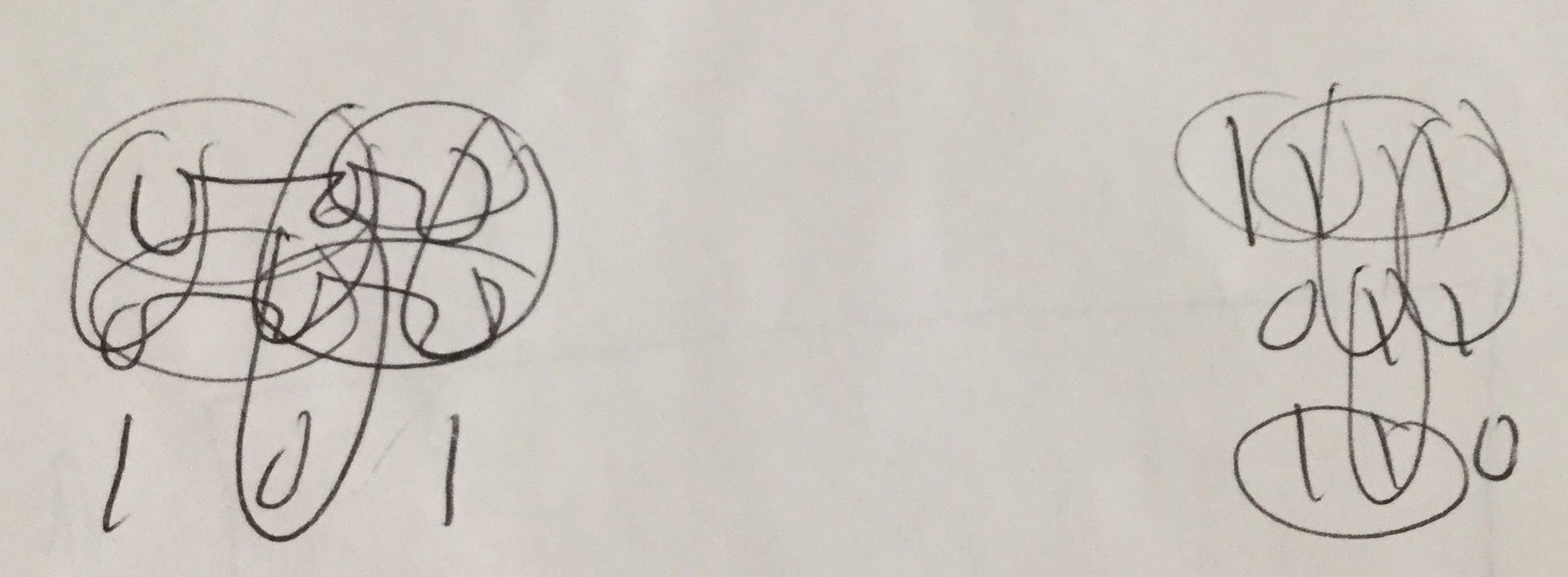}  
\end{center} 
$X_{left} = 8,  \;\ X_{right} = 7$.

We want to find $E[X^r]$ and $E[(X-\mu)^r]$ 
for a fixed $r$ but general $m,n.$ 

The simplest case, $E[X]$ for a 2-by-$n$ board is
\[  \frac{1}{2},  \frac{4}{2},  \frac{7}{2},  \frac{10}{2},  \frac{13}{2},  \frac{16}{2},
 \frac{19}{2}, \dots  \]
We see that $E[X_{(2,n)}]= \dfrac{3n-2}{2}.$  

The data of $2^{mn} \cdot E[X] $ from the program are 
\begin{center}
\begin{tabular}{ c|cccccccc }
  $m : n$ & 1& 2 & 3 &4& 5 & 6&7 &8 \\ \hline 
  1 & 0& 2& 8& 24& 64& 160& 384& 896 \\
  2 & 2 & 32& 224& 1280& 6656& 32768& 155648 & 720896 \\
  3 & 8 & 224& 3072& 34816& 360448 \\
  4 &  24 &1280& 34816& 786432
\end{tabular}
\end{center}

These numbers are $2^{mn-1}(2mn-m-n).$ 
Hence $\mu := E[X_{(m,n)}] = mn-\frac{m}{2}-\frac{n}{2}.$
This observation can be proved combinatorially too. 

Let's move to the second moment. 
The data of $2^{mn} \cdot E[X^2]$ are 

\begin{center}
\begin{tabular}{ c|cccccccc }
  $m : n$ & 1& 2 & 3 &4& 5 & 6&7 &8 \\ \hline 
  1 &  0&2& 12& 48& 160& 480& 1344& 3584\\
  2 &  2& 80& 896& 7040& 46592& 278528& 1556480& 8290304\\
  3 &  12& 896& 19968& 313344& 4145152\\
  4 &  48& 7040& 313344& 9830400
\end{tabular}
\end{center}

These numbers are $2^{mn}(mn-m/2-n/2+1/2)(mn-m/2-n/2)= 2^{mn}\mu(\mu+1/2).$ 
Hence $E[X^2]=  \mu(\mu+1/2) = \mu^2+\frac{\mu}{2}.$
It follows that $Var = E[X^2]-E[X]^2 = \frac{\mu}{2}  $. 
This observation can be shown formally using the overlapping stage method.
\begin{align*}
E[X^2] &= \sum_{i,j} E[X_iX_j] = \sum_i E[X_i^2] + \sum_{i,j: i\neq j}E[X_iX_j]  \\
&= \frac{1}{2}A+ \frac{1}{4}A(A-1) = \frac{A^2}{4}+\frac{A}{4} \\
&= \mu^2+\frac{\mu}{2} ,
\end{align*}
where $A = 2mn-m-n$, the number of possible ways to 
put the domino on an $m$-by-$n$ board.

Next is the third moment. 
The data of $2^{mn} \cdot E[X^3]$ are 

\begin{center}
\begin{tabular}{ c|cccccccc }
  $m : n$ & 1& 2 & 3 &4& 5 & 6&7 &8 \\ \hline 
  1 &  0& 2& 20& 108& 448& 1600& 5184& 15680 \\
  2 &  2& 224& 3920& 41600& 346112& 2490368& 16265216& 99123200\\
  3 &  20& 3920& 138240& 2959360& 49561600\\
  4 &  108& 41600& 2959360& 127401984
\end{tabular}
\end{center}

These numbers are $2^{mn}\mu^2(\mu+3/2).$ 
Then $E[X^3] =\mu^2(\mu+3/2).$
Then it follows that $E[(X-\mu)^3] = E[X^3]-3E[X^2]E[X]+2E[X]^3 = 0  $. 
For formal proof, we apply the overlapping stage approach.
\begin{align*}
E[X^3] = \sum_{i,j,k} E[X_iX_jX_k] &= \sum_i E[X_i^3] 
+ \sum_{i,j: i\neq j}E[X_i^2X_j] + \sum_{i,j,k: i\neq j\neq k}E[X_iX_jX_k]  \\
&= \frac{1}{2}A+ \frac{1}{4}(3)A(A-1) +\frac{1}{8}A(A-1)(A-2)  \\
&= \mu^3+\frac{3\mu^2}{2} ,     
\end{align*}
where again $A=2\mu.$

The higher moments can be calculated directly by
\[     E[X^r] = \sum_{i=1}^r {r \brace i} \dfrac{(A)_i}{2^i} ,\]
where $(A)_i= A(A-1)(A-2)\dots(A-i+1).$

Some examples are
\begin{align*}
E[X^4] &= \dfrac{\mu(2\mu+1)(2\mu^2+5\mu-1)}{4},  \\
E[X^5] &= \dfrac{\mu^2(4\mu^3+20\mu^2+15\mu-5)}{4},  \\
E[X^6] &= \dfrac{\mu(2\mu+1)(4\mu^4+28\mu^3+31\mu^2-23\mu+4)}{8}.
\end{align*}

The moments about the mean are
\begin{align*}
 E[(X-\mu)^{2r+1}] &= 0,   \;\  r=0,1,2,..., \\
Var := E[(X-\mu)^2] &= \dfrac{\mu}{2}, \\
E[(X-\mu)^4] &= \dfrac{\mu(3\mu-1)}{4}, \\
E[(X-\mu)^6] &= \dfrac{\mu(15\mu^2-15\mu+4)}{8}.
\end{align*}

\subsection{Asymptotic Normality of $X$ on the Board of size 1-by-$n$}

Consider board of size 1-by-$n$.
We will work toward the general formula of the leading terms of $E[(X-\mu)^r]$
and use it to show normality of $X.$

It is easy to see that the probability generating function 
$F_n(q)$ is $\left(\dfrac{1+q}{2}\right)^{n-1}.$ 
The centralized probability generating function 
\[ G_n(q) = \dfrac{(1+q)^{n-1}}{q^{(n-1)/2}2^{n-1}}  = \left(\dfrac{q^{1/2}+q^{-1/2}}{2}\right)^{n-1}. \] 

We will use the same technique as in the inversion number section
of getting a recurrence. We start with the fact that
\begin{equation} \label{RecDo}
 G_n(1+z) = \sum_{r=0}^{\infty} B_r(n)z^r,
\end{equation}
where $B_r(n) := E[\binom{X}{r}]$, the binomial moment. 

Also notice that 
\[P_n(q) := \dfrac{G_n(q)}{G_{n-1}(q)} = \dfrac{q^{1/2}+q^{-1/2}}{2}.  \]
Hence,
\[  P_n(1+z) = \dfrac{2+z}{2(1+z)^{1/2}} = 
1+\frac{1}{8}z^2-\frac{1}{8}z^3+\frac{15}{128}z^4
-\frac{7}{64}z^5+ \dots .    \]

Then we form the recurrence using \eqref{RecDo} and $G_n(1+z)=P_n(1+z)G_{n-1}(1+z)$:
\begin{equation} \label{RecBDo}  
B_r(n) = B_r(n-1)+\frac{1}{8}B_{r-2}(n-1)-\frac{1}{8}B_{r-3}(n-1)+\dots .
\end{equation}

With the help of computer, we conjecture the formula of the leading term 
of $B_r(n)$ (general $r$) and 
apply the induction on $r$ to \eqref{RecBDo} to prove the formula. 
 
The Taylor expansion of $G_n(1+z)$ around $z=0$ looks like

\[1+\dfrac{(n-1)}{8}z^2-\dfrac{(n-1)}{8}z^3
+\dfrac{(n-1)(n+13)}{128}z^4
-\dfrac{(n-1)(n+5)}{64}z^5+\dots \]

We conjecture that
\begin{align*}
B_{2r}(n) &= \dfrac{n^{r}}{r!2^{3r}} + \;\ \mbox{ lower order terms ,} \\
B_{2r+1}(n) &= -\dfrac{n^{r}}{(r-1)!2^{3r}} + \;\ \mbox{ lower order terms.}
\end{align*}

Induction using equation \eqref{RecBDo} goes through without a problem
(the reader is invited to check that themselves). 
Finally we turns these binomial moment to straight moment by 
a Stirling number of the second kind, i.e. $E[X^r] = \sum_i {r \brace i}B_i(n) \cdot i!.$
We see that 
\begin{align*}
E[(X-\mu)^{2r}] &= \dfrac{1}{r!2^{3r}}\cdot (2r)!n^r + \;\ \mbox{ lower order terms ,} \\
E[(X-\mu)^{2r+1}] &= 
\left(-\dfrac{1}{(r-1)!2^{3r}}\cdot (2r+1)! +  \frac{(2r+1)(2r)}{2} \dfrac{1}{r!2^{3r}}\cdot(2r)!  \right) n^r
+  \;\ \mbox{ lower order terms ,}  \\
&= 0 \cdot n^r + \;\ \mbox{ lower order terms.}
\end{align*}

Hence, as $n \to \infty,$
\[ \dfrac{E[(X-\mu)^{2r}]}{Var^r} =   \dfrac{(2r)!}{r!2^{r}}    \;\
\mbox{ and } \;\ \dfrac{E[(X-\mu)^{2r+1}]}{Var^{r+1/2}} = 0, \] 
and the normality of this case is proved.

\begin{rem}
We can do better by conjecture (and prove) more leading terms, i.e.
we claim that
\begin{align*}
B_{2r}(n) &= \dfrac{n^{r}}{r!2^{3r}} +\dfrac{r(4r-5)n^{r-1}}{(r-1)!2^{3r}}
+ \;\ \mbox{ lower order terms ,} \\
B_{2r+1}(n) &= -\dfrac{n^{r}}{(r-1)!2^{3r}} -\dfrac{r(4r^2-3r-4)n^{r-1}}{3(r-1)!2^{3r}}
+\;\ \mbox{ lower order terms.}
\end{align*}
\end{rem}

\subsection{Asymptotic Normality of $X$ on the Board of size $m$-by-$n$}

For any fix $m$, the normality of $X$, a random number of dominoes,
on the board of size $m$-by-$n$ as $n \to \infty$ 
follows from section 6.1 and the discussion before it. 
We learned that, for each $r$, $E[X^r]$ are all the same 
when written as a function of $\mu$ 
for any $m$ and $n.$ Then the result follows from that 
$X$ of board of size 1-by-$n$ is asymptotically normal.
We note that, although the moments are the same for different $m$ and $n$, 
the generating functions for a fixed $m\geq 2$
are more complicated and don't seem to have nice patterns.

\section{The Moment Generating Function, $\phi(t)$}

In this section, we show the normality result of 
inversion numbers/major index and Boolean board of size 1-by-$n$ 
again but with the more direct method. We show that 
\[\lim_{n \to \infty} G_n(e^{t/{\sigma}}) = e^{t^2/2}.\]
The outline of this method was already mentioned 
in section 3.

The generating function for Boolean functions case is more complicated.
We found a good simpler approximating generating function 
and try to show the normality of that approximating generating function instead. 
\subsection{Inversion Numbers/Major Index}
The centralized generating function of inversion numbers/major index is
\[   G_n(q) = \dfrac{1}{n!q^{n(n-1)/4}}\prod_{i=1}^n \dfrac{1-q^i}{1-q} 
=  \dfrac{1}{n!}\prod_{i=1}^n \dfrac{q^{i/2}-q^{-i/2}}{q^{1/2}-q^{-1/2}}.\] 

We consider
\begin{align*}
\phi(t) := G_n(e^{t/{\sigma}}) &=  
\dfrac{1}{n!}\dfrac{  \prod_{i=1}^n \left(\dfrac{e^{it/2\sigma}-e^{-it/2\sigma} }{2}\right)}
{ \left(\dfrac{e^{t/2\sigma}-e^{-t/2\sigma}}{2}\right)^n}  \\
&=     \dfrac{1}{n!}\dfrac{  \prod_{i=1}^n 
{ \left(\dfrac{it}{2\sigma}+\dfrac{1}{3!}\left(\dfrac{it}{2\sigma}\right)^3 
+\dfrac{1}{5!}\left(\dfrac{it}{2\sigma}\right)^5+\dots\right)}}
{ \left(\dfrac{t}{2\sigma}+\dfrac{1}{3!}\left(\dfrac{t}{2\sigma}\right)^3 
+\dfrac{1}{5!}\left(\dfrac{t}{2\sigma}\right)^5+\dots\right)^n}  \\   
&=  \dfrac{  \prod_{i=1}^n 
{ \left(1+\dfrac{1}{3!}\left(\dfrac{it}{2\sigma}\right)^2 
+\dfrac{1}{5!}\left(\dfrac{it}{2\sigma}\right)^4
+\dots\right)}}
{ \left(1+\dfrac{1}{3!}\left(\dfrac{t}{2\sigma}\right)^2
+\dfrac{1}{5!}\left(\dfrac{t}{2\sigma}\right)^4+\dots\right)^n}  
\end{align*} 

Next we consider ${n \to \infty}$
and apply the fact that $\sigma^2= \dfrac{n(n-1)(2n+5)}{72}$  
on both numerator and denominator.
First we note that 
\[  \lim_{n \to \infty} \sum_{i=1}^n 
{ \left(\dfrac{1}{3!}\left(\dfrac{it}{2\sigma}\right)^2 
+\dfrac{1}{5!}\left(\dfrac{it}{2\sigma}\right)^4
+\dots\right)}  = \dfrac{t^2}{2} ,  \]
and
\[  \lim_{n \to \infty}  
 n\left(\dfrac{1}{3!}\left(\dfrac{t}{2\sigma}\right)^2 
+\dfrac{1}{5!}\left(\dfrac{t}{2\sigma}\right)^4
+\dots\right)  = 0 .  \]

Lastly, since  $\ln(\prod(1+x_i)) = \sum \ln(1+x_i) 
\approx \sum x_i,$ when $x_i$ is small, we have
\[ \lim_{n \to \infty} G_n(e^{t/{\sigma}}) =
\lim_{n \to \infty} \dfrac{  \prod_{i=1}^n 
{ \left(1+\dfrac{1}{3!}\left(\dfrac{it}{2\sigma}\right)^2 
+\dfrac{1}{5!}\left(\dfrac{it}{2\sigma}\right)^4
+\dots\right)}}
{ \left(1+\dfrac{1}{3!}\left(\dfrac{t}{2\sigma}\right)^2
+\dfrac{1}{5!}\left(\dfrac{t}{2\sigma}\right)^4+\dots\right)^n}  
=\dfrac{e^{t^2/2}}{e^0} = e^{t^2/2}.\]

Therefore the distribution of inversion numbers 
on the permutation of length $n$ is normal as $n \to \infty$.

\subsection{Boolean Functions }  \label{Exp}

We consider asymptotic normality of number 
of $k$-dimensional cube over the space of Boolean 
functions of length $n$ ($n \to \infty$). We learned
from section \ref{BFMo} that it is very difficult to calculate
the moment and the generating functions of these numbers. 
So we have an idea that we assume the independent
between each objects $X$ that we count and use it
as an approximation of the actual numbers.
\subsubsection*{Experimental Math. on Generating Functions} 
We let $Y$ be the random number of this 
$k$-dimensional cube that we mentioned. We want
to come up with its generating function and test the approximation.

We think of $m$ as a number of elements in Boolean function $f$. 
Let $p_k$ be the probability that the randomly chosen $2^k$ elements, 
each with length $n$, forms a $k$-cube. 
\[ p_k=\binom{n}{k}\dfrac{(2^k)!}{2^k2^{n(2^k-1)}}.  \]
For examples, $p_0=1$ and $p_1=\dfrac{n}{2^n}.$

The probability generating function $H_n(q)$ is
\begin{align*} 
H_n(q) &=  \dfrac{1}{2^{2^n}}\sum_{m=0}^{2^n}
\binom{2^n}{m}  \sum_{s=0}^{M} 
\binom{M}{s}p^s(1-p)^{M-s}q^s  , \;\  \mbox{where } M = \binom{m}{2^k}\\
&=  \dfrac{1}{2^{2^n}}\sum_{m=0}^{2^n}
\binom{2^n}{m}(pq+1-p)^{M} .
\end{align*}

We show the distributions of numbers of $1$-cube ($k=1$) 
with boolean function of length $n=4$
where the blue bars are from empirical data and the aqua bars are from $H_n(q).$
It looks quite close even with small $n.$
\begin{center}
\includegraphics[scale=0.45]{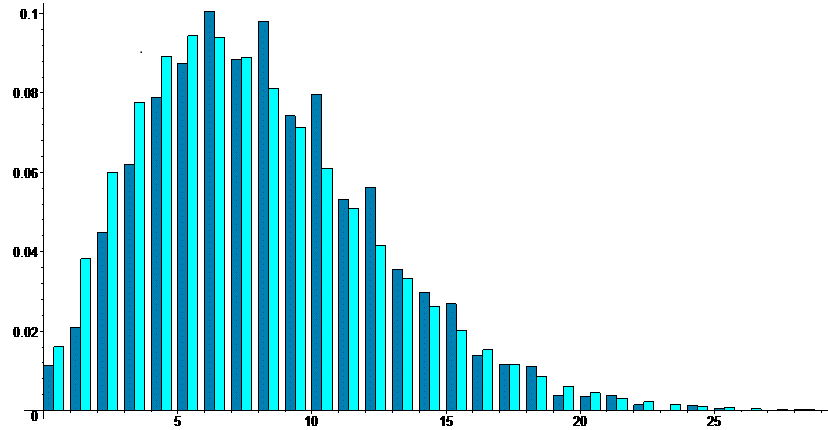}  
\end{center}

\subsubsection*{Experimental Math. on Moments} 
We also compare the actual average
$E[X]= \binom{n}{k}\dfrac{2^{n-k}}{2^{2^k}}$
with $E[Y] = H'_n(q)|_{q=1}$ for $k=0,1,2.$
\begin{align*} 
k=0, \;\ E[X]=2^{n-1}, \;\ E[Y] &= 
\dfrac{1}{2^{2^n}}\sum_{m=0}^{2^n} \binom{2^n}{m}\cdot m = 2^{n-1}. \\
k=1, \;\ E[X]=\dfrac{n2^n}{8}, \;\  E[Y] &= 
\dfrac{1}{2^{2^n}}\sum_{m=0}^{2^n} \binom{2^n}{m}\cdot \binom{m}{2} 
= \dfrac{n(2^n-1)}{8}. \\
k=2, \;\ E[X]=\binom{n}{2}\dfrac{2^n}{64}, \;\  E[Y] &= 
\dfrac{1}{2^{2^n}}\sum_{m=0}^{2^n} \binom{2^n}{m}\cdot \binom{m}{4} 
= \binom{n}{2}\dfrac{2^n-6+(11)2^{-n}-(6)2^{-2n})}{64}. 
\end{align*}
For the variances
\begin{align*} 
k=0, \;\ Var[X] &=\dfrac{2^n}{2^2}, \;\ Var[Y] =  \dfrac{2^n}{2^2}. \\
k=1, \;\ Var[X] &=\dfrac{n(2n+1)2^n}{2^5}, \;\  
\mathcal{L}\{Var[Y]\} = \dfrac{n[(2n+4)2^n]}{2^5}. \\
k=2, \;\ Var[X] &=\dfrac{n(n-1)(2n^2+2n+3)2^n}{2^{11}}, \\  
\mathcal{L}\{Var[Y]\} &= \dfrac{n(n-1)(n^2-n+8)2^n}{2^{10}}. 
\end{align*}
Their leading terms are the same.

\subsubsection*{Experimental Math. on Normality} 
Some leading terms for case $k=1$ of 
$E[(Y-\mu)^r], r=0,1,2,\dots$ are
\[  1, 0, \dfrac{n^22^n}{16}, \dfrac{3n^32^n}{64}
, \dfrac{3n^44^n}{256}, \dfrac{15n^54^n}{512}, \dfrac{15n^68^n}{4096}
, \dfrac{315n^78^n}{16384}, \dfrac{105n^816^n}{65536} ,... \]

Some leading terms for case $k=2$ of 
$E[(Y-\mu)^r], r=0,1,2,\dots$ are
\[  1, 0, \dfrac{n^42^n}{2^{10}}, \dfrac{9n^62^n}{2^{15}}, 
\dfrac{3n^84^n}{2^{20}}, \dfrac{45n^{10}4^n}{2^{24}}
,\dfrac{15n^{12}8^n}{2^{30}},... \]

These moments satisfy the condition that asymptotically 
$\dfrac{E[(Y-\mu)^{2r}]}{\sigma^{2r}} = \dfrac{(2r)!}{2^rr!} $
and  $\dfrac{E[(Y-\mu)^{2r+1}]}{\sigma^{2r+1}} = 0. $
We realize that the distributions for $k=1$ and $k=2$ 
converge to normal once again.

There are overwhelming evidences that the distribution of 
a random variable $Y$ (as does $X$) should converge to normal.
Although we did not quite manage to show it from the generating
$H_n(q)$ that we defined at the beginning of the section.

\subsection{Boolean Board of Size 1-by-$n$  }

The centralized generating function of 1-by-$n$ boolean board is
\[ G_n(q) = \left(\dfrac{q^{1/2}+q^{-1/2}}{2}\right)^{n-1}.\]
Given that $\sigma^2 = \dfrac{n-1}{4}$, then
\begin{align*} \label{Hot}
  G_n(e^{t/\sigma}) &= \left(\dfrac{e^{\frac{t}{2\sigma}}+e^{\frac{-t}{2\sigma}}}{2}\right)^{n-1} \\
  &=  \left(1+\dfrac{1}{2!}\left(\dfrac{t}{2\sigma}\right)^2 
+\dfrac{1}{4!}\left(\dfrac{t}{2\sigma}\right)^4+\dots\right)^{n-1} \\ 
&=  \left(1+\dfrac{1}{2!}\dfrac{t^2}{n-1} 
+\dfrac{1}{4!}\dfrac{t^4}{(n-1)^2}+\dots\right)^{n-1} \\
&= e^{t^2/2}  \;\  \mbox{ as } n \to \infty.
 \end{align*}
Therefore the distribution of number of 2-by-1 dominoes 
on the boolean board of size 1-by-$n$ is normal as $n \to \infty$.

\end{document}